\documentclass[11pt]{amsart}
\usepackage{amssymb,amsmath,amsthm,newlfont,enumerate}

\theoremstyle{plain}
\newtheorem{theorem}{Theorem}[section]

\newtheorem{lemma}[theorem]{Lemma}

\newtheorem{example}[theorem]{Example}

\theoremstyle{remark}
\newtheorem*{remark}{Remark}
\newtheorem*{remarks}{Remarks}

\newcommand{\CC}{\mathbb{C}}
\newcommand{\RR}{\mathbb{R}}
\newcommand{\ZZ}{\mathbb{Z}}
\newcommand{\cH}{\mathcal{H}}
\newcommand{\cL}{\mathcal{L}}
\newcommand{\cP}{\mathcal{P}}
\renewcommand{\hat}{\widehat}

\begin{document}
	
\title{A Cram\'er--Wold theorem for mixtures}
	
\author[R. Fraiman]{Ricardo Fraiman}
\address{Centro de Matem\'atica, Facultad de Ciencias, Universidad de la Rep\'ublica, Uruguay.} 
\email{rfraiman@cmat.edu.uy} 
	
\author[L. Moreno]{Leonardo Moreno}
\address{Instituto de Estad\'{i}stica, Departamento de M\'etodos Cuantitativos, FCEA, Universidad de la Rep\'ublica, Uruguay.}
\email{mrleo@iesta.edu.uy} 
	
\author[T. Ransford]{Thomas Ransford}
\address{D\'epartement de math\'ematiques et de statistique, Universit\'e Laval, Qu\'ebec City (Qu\'ebec),  Canada G1V 0A6.}
\email{ransford@mat.ulaval.ca}
	
\thanks{Fraiman and Moreno were   supported by grant FCE-3-2022-1-172289, Ag\-encia Nacional de Investigaci\'on e Innovaci\'on, Uruguay. Ransford  was supported by NSERC Discovery Grant RGPIN--2020--04263.}

\begin{abstract}
We show how a Cram\'er--Wold  theorem for a family of multivariate probability
distributions  can be used to generate a similar theorem for mixtures (convex combinations)
of distributions drawn from the same family.
		
Using this abstract result, we establish a Cram\'er--Wold theorem for mixtures of multivariate Gaussian distributions.
According to this theorem, two such mixtures can be distinguished by projecting them onto a certain predetermined finite set of lines, the number of lines depending only on the total number Gaussian distributions  involved and on the ambient dimension. A similar result is also obtained for mixtures of multivariate $t$-distributions.
\end{abstract}
	
\keywords{Projection, identifiability,  mixture, Gaussian distribution,  $t$-distribution}
	
\makeatletter
\@namedef{subjclassname@2020}{\textup{2020} Mathematics Subject Classification}
\makeatother
	
\subjclass[2020]{60B11}
	
\maketitle

\section{Introduction}
	
The Cram\'er--Wold device is the name given to a general
technique for analyzing multivariate probability distributions 
via their lower-dimensional projections. The name originates
from a classical theorem of Cram\'er and Wold \cite{CW36} to the effect that
a probability measure in Euclidean $d$-dimensional space is 
uniquely determined by its one-dimensional projections in all directions.
	
For certain classes of measures, one can do better.
Indeed, in some cases,  just finitely many projections suffice to determine the measure. 
As an example, we cite the
case of elliptic measures in $\RR^d$
(which includes that of Gaussian measures), 
where just $(d^2+d)/2$ suitably chosen  projections suffice (see \cite{FMR23a}).
	
The central object of study in this article is that of \emph{mixtures}, 
namely convex combinations of probability measures. 
For example, suppose that we already have a Cram\'er--Wold device 
for a certain  family of probability measures: 
can we then deduce a similar device for mixtures of measures taken from that family? 
We give an affirmative answer to this question in certain cases, 
including the very important one of Gaussian mixtures. 
	
Here, in more detail, is a road-map of the article.	
	
Section~\ref{S:mixtures} begins with a short introduction to the notion of the Cram\'er--Wold device.
We then establish a rather general abstract result, Theorem~\ref{T:mixture},
which shows how a Cram\'er--Wold theorem 
for a class of measures can be used to generate a similar theorem for mixtures drawn from the same class.
	
To apply Theorem~\ref{T:mixture},
it is necessary to check that certain families of
probability distributions are linearly independent, or equivalently,
that they are identifiable. There is an extensive literature
on techniques for proving identifiability.
In Section~\ref{S:lindep}, we review results concerning the identifiability of 
two particular families: the Gaussian distributions and the $t$-distributions.
We also derive what we believe to be a new  criterion for identifiability,
based on a generalization of Picard's theorem from complex analysis due to Borel.
	
In Section~\ref{S:Gaussian},
these ideas are  applied to the  important
special case of Gaussian measures and Gaussian mixtures. 
In particular, we derive
a Cra\-m\'er--Wold theorem for Gaussian mixtures,
Theorem~\ref{T:gaussmixture}, from the previously known one for Gaussian measures.
It says that, if $P$ and $Q$ are both mixtures of $m$
normal distributions on $\RR^d$,
and if their
one-dimensional projections  coincide 
on a certain pre-determined set of $(2m-1)(d^2+d-2)/2+1$ lines, 
then $P$ and $Q$ are the same. 
A similar result, Theorem~\ref{T:tmixture}, is also established for mixtures of multivariate
$t$-distributions.
	
We conclude in Section~\ref{S:conclusion} with some remarks
about applications of these results.

	
\section{A Cram\'er--Wold theorem for mixtures}\label{S:mixtures}
	
\subsection{Introduction}
	
Given a Borel probability measure $P$ on $\RR^d$ and a vector subspace $H$ of $\RR^d$,
we write $P_H$ for the projection of $P$ onto $H$, namely the Borel probability measure on $H$ given by
\[
P_H(B):=P(\pi_H^{-1}(B)),
\]
where $B$ is an arbitrary Borel subset of $H$
and $\pi_H:\RR^d\to H$ is the orthogonal projection of $\RR^d$ onto $H$.
	
According to a well-known theorem of Cram\'er and Wold \cite{CW36}, 
if $P,Q$ are two Borel probability measures on $\RR^d$, 
and if $P_L=Q_L$ for all lines $L$, then $P=Q$.
	
There have been numerous refinements of the Cram\'er--Wold theorem.
To help describe these, it is convenient to introduce the following terminology.
Let  $\cP$ be a family of Borel probability measures on $\RR^d$ 
and let $\cH$ be a family of vector subspaces of $\RR^d$ (not necessarily all of the same dimension).
We say that $\cH$ is a \emph{Cram\'er--Wold system}  for $\cP$ if, for every pair $P,Q\in\cP$,
\[
P_H=Q_H ~(\forall H\in \cH)\quad\Rightarrow\quad P=Q.
\]
In this terminology, the original Cram\'er--Wold theorem says 
simply that the family of all lines in $\RR^d$ is a Cram\'er--Wold
system for the family of  Borel probability measures on $\RR^d$.
Here are some other examples.
	
\begin{itemize}
\item If $\cP$ is the family of compactly supported Borel probability measures on $\RR^2$,
then any infinite set  of lines in $\RR^2$ is a Cram\'er--Wold system for $\cP$ 
(R\'enyi \cite[Theorem~1]{Re52}).
		
\item If $\cP$ is the family of probability measures on $\RR^d$
whose supports contain at most $k$ points, and if 
$\cH=\{H_1,\dots,H_{k+1}\}$ is any family of $k+1$ subspaces of  $\RR^d$ such that 
$H_i^\perp\cap H_j^\perp=\{0\}$ whenever $i\ne j$,
then $\cH$ is a Cram\'er--Wold system for $\cP$
(Heppes, \cite[Theorem~$1'$]{He56}).
		
\item If $\cP$ is the family of probability measures $P$ on $\RR^d$
whose moments $m_n:=\int_{\RR^d}\|x\|^n\,dP(x)$ are finite and satisfy $\sum_n m_n^{-1/n}=\infty$, 
and if $\cH$ is any family of subspaces such that 
$\cup_{H\in\cH}H$ has positive Lebesgue measure in $\RR^d$,
then $\cH$ is a Cram\'er--Wold system for $\cP$ (Cuesta-Albertos \textit{et al} \cite[Corollary~3.2]{CFR07}).
		
\item If $\cP$ is the family of elliptical distributions on $\RR^d$ and if
$\cL$ is the set of  lines  $\langle e_i+e_j\rangle ~(1\le i\le j\le d)$,
where $\{e_1,\dots,e_d\}$ is any basis of $\RR^d$,
then $\cL$ is a Cram\'er--Wold system for $\cP$
(Fraiman \textit{et al}, \cite[Theorem~1]{FMR23a}).
		
\end{itemize}

Further results of this kind may be found in \cite{BMR97, FMR23b, FMR24, Gi55,  GJ20}.
	
	
\subsection{Mixtures}
	
Let $d\ge1$ and let $\cP$ be a family of Borel probability measures on $\RR^d$.
	
We say that $\cP$ is \emph{linearly dependent} if there exist distinct  $P_1,\dots,P_n\in\cP$
and non-zero  $\lambda_1,\dots,\lambda_n\in\RR$ such that $\sum_{j=1}^n\lambda_j P_j=0$.
Otherwise $\cP$ is \emph{linearly independent}.
	
A \emph{$\cP$-mixture}  is a  convex combination of measures from $\cP$, 
in other words, a measure of the form $\sum_{j=1}^n\lambda_jP_j$, 
where $P_1,\dots,P_n\in\cP$ and $\lambda_1,\dots,\lambda_n\ge0$ 
with $\sum_{j=1}^n\lambda_j=1$.
	
Our aim in this section is to establish the following Cram\'er--Wold  theorem for $\cP$-mixtures.
	
\begin{theorem}\label{T:mixture}
Let $\cP$ be a family of Borel probability measures on $\RR^d$,
let $\cH$ be a collection of vector subspaces of $\RR^d$, and let $m\ge1$.
Suppose that:
\begin{enumerate}[\normalfont\rm(i)]
\item\label{i:1} 
For each $H\in\cH$, the distinct measures in the set
$\{P_H:P\in\cP\}$ are linearly independent.
\item\label{i:2} 
Each partition of $\cH$ into $2m-1$ subsets contains at least one Cram\'er--Wold system for $\cP$.
\end{enumerate}
Let $P$ and $Q$ each be convex combinations of $m$ measures from $\cP$.
If $P_H=Q_H$ for all $H\in\cH$, then $P=Q$.
\end{theorem}

\begin{remark}
The statement of Theorem~\ref{T:mixture} immediately
raises the question of how to verify the hypotheses (i) and~(ii).
Hypothesis~(i) will be treated in detail Section~\ref{S:lindep}.
Hypothesis~(ii) will be addressed (at least in the special case of 
Gaussian and $t$-distributions) in Section~\ref{S:Gaussian}.
\end{remark}
	
\begin{proof}[Proof of Theorem~\ref{T:mixture}]
We argue by contradiction. 
Suppose that $P_H=Q_H$ for all $H\in\cH$, but that $P\ne Q$. 
Then the difference $P-Q$ can be written as 
$P-Q=\sum_{j=1}^k\lambda_j P_j$, 
where $P_1,\dots,P_k\in\cP$ are distinct,  
where $\lambda_1,\dots,\lambda_k\in\RR$  are non-zero,
and where $k\le 2m$. Set
\[
\cH_j:=\{H\in\cH: (P_{1})_H= (P_{j})_H\} \quad(j=2,\dots,k).
\]
Clearly no $\cH_j$ is a Cram\'er--Wold system for $\cP$.
If we had $\cup_{j=2}^k \cH_j=\cH$, then we could construct a partition
of $\cH$ into $k-1$ sets (perhaps some of them empty) none of which is
a Cram\'er--Wold system for $\cP$, contradicting the assumption~\eqref{i:2}
on $\cH$.
We conclude that $\cup_{j=2}^k \cH_j\ne\cH$, so there exists an $H_0\in\cH$
such that $(P_{1})_{H_0}\ne (P_{j})_{H_0}~(j=2,\dots,k)$.
		
By assumption~\eqref{i:1},
the distinct measures in the set $\{(P_j)_{H_0}:j=1,\dots k\}$
are linearly independent. In particular, since $(P_{1})_{H_0}\ne (P_{j})_{H_0}~(j=2,\dots,k)$,
it follows that $(P_{1})_{H_0}$ is not in the span of $\{(P_j)_{H_0}:j=2,\dots d\}$. 
On the other hand, we have
\[
\sum_{j=1}^k\lambda_j(P_j)_{H_0}
=\Bigl(\sum_{j=1}^k\lambda_j P_j\Bigr)_{H_0}=(P-Q)_{H_0}=P_{H_0}-Q_{H_0}=0.
\]
Since $\lambda_1\ne0$, this is a contradiction.
\end{proof}

	
\section{Linear independence and identifiability}\label{S:lindep}
		
\subsection{Introduction}

Condition \eqref{i:1} in Theorem~\ref{T:mixture} 
begs the question as to how one determines 
whether a set of measures on a subspace $H$ of $\RR^d$ is linearly independent. 
		
In the statistical literature, the notion of linear independence of measures
is synonymous with that of identifiability.
We say that a family $\cP$ of Borel probability measures on $\RR^d$ is \emph{identifiable}
if it is impossible to express any $\cP$-mixture as two different convex
combinations of elements of $\cP$.
It is easy to see that $\cP$ is identifiable if and only if it
is linearly independent (see e.g.\ \cite[Theorem~3.1.1]{TSM85}).
There is an extensive literature concerning techniques for proving identifiability/linear independence.
A useful background reference is the book \cite{TSM85}.

\subsection{Linear independence of Gaussian and $t$-distributions}
	
We shall need two results in particular. The first concerns the family 
of Gaussian distributions, with densities
\begin{equation}\label{E:Gaussdf}
f_{\mu,\sigma}(x)=\frac{1}{\sqrt{2\pi\sigma^2}}\exp\Bigl(-\frac{(x-\mu)^2}{2\sigma^2}\Bigr) \quad(x\in\RR).
\end{equation}
		
\begin{theorem}\label{T:normalindep}
The family of Gaussian distributions $\{f_{\mu,\sigma}:\mu\in\RR,\sigma>0\}$ on $\RR$ is 
linearly independent.
\end{theorem}
		
\begin{proof}
This result is well known. 
An early reference is  \cite[Proposition~2]{YS68}. See also  \cite[Example 3.1.4]{TSM85}.
\end{proof}
		
The second result treats the identifiability/linear independence of
the family of $t$-distributions. Recall that 
a \emph{$t$-distribution} on $\RR$ with $\nu$ degrees of freedom is a Borel measure with density of the form
\[
f_{\nu,\mu,\sigma}(x)=c_{\nu,\mu,\sigma}\Bigl(1+\frac{(x-\mu)^2}{\nu \sigma^2}\Bigr)^{-(\nu+1)/2},
\]
where $\nu$ is a positive integer, $\mu\in\RR$ and $\sigma>0$.
The constant $c_{\nu,\mu,\sigma}$ is chosen to ensure that 
$\int_\RR f_{\nu,\mu,\sigma}(x)\,dx=1$. 
This  distribution has mean $\mu$ (if $\nu>1$)
and variance $\sigma^2\nu/(\nu-2)$ (if $\nu>2$).
		
\begin{theorem}\label{T:tlindep}
The family of $t$-distributions $\{f_{\nu,\mu,\sigma}:\nu\in\ZZ^+, \mu\in\RR,\sigma>0\}$
is linearly independent.
\end{theorem}

\begin{proof}
See for example \cite[Section~3, Example~1]{HMG06}.
\end{proof}

		
\subsection{Linear independence of measures via Borel's theorem}
Theorems~\ref{T:normalindep} and \ref{T:tlindep} above are just two examples of a 
large class of theorems asserting the identifiability of various families of distributions.
The proofs of many of these results are very similar, 
and there is even an abstract
formulation of the general method due to Teicher \cite[Theorem~2]{Te63}.

In this subsection we present
a further general result on identifiability/linear independence,
based on a completely different principle,
namely a generalization of Picard's theorem from complex analysis due to Borel \cite{Bo1897}.
We have not seen this theorem used elsewhere in the statistical literature,
and  we believe that the following application may be of interest.
		
\begin{theorem}\label{T:lindep}
Let $\cP$ be a family of Borel probability measures on $\RR^d$.
Suppose that the characteristic function of every measure in $\cP$  
is the restriction to $\RR^d$ of a nowhere-vanishing holomorphic function on $\CC^d$.
Then $\cP$ is linearly independent.
\end{theorem}
		
\begin{remarks}
(i) This result immediately yields  an alternative proof of Theorem~\ref{T:normalindep},
since the characteristic functions of the Gaussian distributions
\eqref{E:Gaussdf} have the form $\phi(\xi)=\exp(i\mu \xi-\sigma^2 \xi^2/2)$.
			
(ii) There is a well-known criterion for the characteristic function of $P$  to be the restriction to $\RR^d$
of a holomorphic function on $\CC^d$: this happens if and only if the moments $m_n:=\int_{\RR^d}\|x\|^n\, dP(x)$
are finite and satisfy  $m_n^{1/n}=o(n)$ as $n\to\infty$ (see e.g.\ \cite[Theorem~4.2.2]{LR79}).
			
(iii) Suppose that characteristic function of $P$ is the restriction to $\RR^d$
of a holomorphic function $f$ on $\CC^d$. A sufficient condition for $f$ to be zero-free on $\CC^d$ is that $P$
be infinitely divisible. This follows from \cite[Theorem~3.1]{HS78} and \cite[Theorem~4]{LS52}.
			
(iv) The following simple example shows that 
the nowhere-vanishing condition in Theorem~\ref{T:lindep} cannot be omitted.
Let $P_0$ and $P_1$ be the Dirac measures on $\RR$ concentrated at $0$ and $1$ respectively, 
and let $P_2:=(1/3)P_0+(2/3)P_1$. 
The characteristic functions are the $P_j$ are given by 
$\phi_{P_0}(\xi)=1$ and $\phi_{P_1}(\xi)=e^{i\xi}$ and $\phi_{P_2}=(1/3)(1+2e^{i\xi})$.
All three are nowhere-zero on $\RR$ and
all three extend to be holomorphic on $\CC$.
However,  the set $\{P_0,P_1,P_2\}$ is clearly linearly dependent. 
Theorem~\ref{T:lindep} does not apply in this case, because the holomorphic extension of $\phi_{P_2}$
has zeros in $\CC$ (namely at $i\log2 +(2n+1)\pi$ for each integer~$n$).
\end{remarks}
		
As mentioned above, the proof of Theorem~\ref{T:lindep} is based on a theorem of Borel. 
The precise version that we need is due to Green \cite[p.98]{Gr72}:
		
\begin{lemma}\label{L:lindep}
Let $g_1,\dots,g_n:\CC^d\to\CC$ be holomorphic functions such that
\[
\exp(g_1)+\cdots+\exp(g_n)\equiv0.
\]
Then, for some distinct $j,k$, the function $g_j-g_k$ is constant.
\end{lemma}
		
\begin{proof}[Proof of Theorem~\ref{T:lindep}]
We argue by contradiction. Suppose that $\cP$ is linearly dependent, so
there exist distinct $P_1,\dots,P_n\in\cP$ and non-zero scalars $\lambda_1,\dots,\lambda_n$
such that $\sum_{j=1}^n\lambda_j P_j=0$. Then the characteristic functions $\phi_{P_j}$ of the $P_j$
satisfy $\sum_{j=1}^n \lambda_j\phi_{P_j}(\xi)=0$ for all $\xi\in\RR^d$.
By assumption, each $\phi_{P_j}$ is the restriction to $\RR^d$ of a nowhere-vanishing
holomorphic function on $\CC^d$. Thus we may write $\lambda_j \phi_{P_j}=\exp g_j|_{\RR^d}$,
where $g_j:\CC^d\to\CC$ is a holomorphic function on $\CC^d$,
and $\sum_{j=1}^n\exp g_j(\xi)=0$ for all $\xi\in\RR^d$.
By the identity principle, a holomorphic function on $\CC^d$ that vanishes on $\RR^d$ 
is identically zero on $\CC^d$. Therefore $\sum_{j=1}^n\exp g_j(\zeta)=0$ for all $\zeta\in\CC^d$.
We now invoke Lemma~\ref{L:lindep}, to deduce that $g_j-g_k$ is constant for some pair of distinct indices $j,k$.
This implies that $\phi_{P_j}/\phi_{P_k}$ is constant on $\RR^d$. 
As $P_j,P_k$ are probability measures,
we have $\phi_{P_j}(0)=\phi_{P_k}(0)=1$. Therefore $\phi_{P_j}=\phi_{P_k}$ on $\RR^d$.
By the uniqueness theorem for characteristic functions, it follows that $P_j=P_k$.
This contradicts the fact that $P_j$ and $P_k$ are distinct measures.
\end{proof}
		
\begin{remark}
Theorem~\ref{T:lindep} actually implies a stronger form of itself, as follows.
Given $A\in M_d(\RR)$ (the set of $d\times d$ matrices)
and $b\in\RR^d$, let us write $P_{A,b}$ 
for the Borel probability measure on $\RR^d$ defined by
\begin{equation}\label{E:affine}
P_{A,b}(B):=P\bigl(\{x\in\RR^d: Ax+b\in B\}\bigr).
\end{equation}
Suppose that the characteristic function of each measure in $\cP$ 
is the restriction to $\RR^d$ of a nowhere-vanishing function on $\CC^d$.
Then, not only is $\cP$ linearly independent, 
but even the distinct measures in the family 
$\{P_{A,b}:P\in\cP, A\in M_d(\RR), b\in\RR^d\}$ are linearly independent.
Indeed, a simple calculation shows that
\[
\phi_{P_{A,b}}(\xi)=e^{ib\cdot \xi}\phi_P(A^T\xi)\quad(\xi\in\RR^d),
\]
so, if $\phi_P$ is the restriction to $\RR^d$ of a nowhere-vanishing
holomorphic function on $\CC^d$, then the same is true of 
$\phi_{P_{A,b}}$. Theorem~\ref{T:lindep} now gives the result.
\end{remark}

	
\section{Gaussian mixtures}\label{S:Gaussian}
	
\subsection{Introduction}
	
A Gaussian measure $P$ on $\RR^d$ is one whose density has the form
\begin{equation}\label{E:density}
\frac{1}{\det(2\pi\Sigma)^{1/2}}\exp\Bigl(-\frac{1}{2}(x-\mu)^T\Sigma^{-1}(x-\mu)\Bigr)
\quad(x\in\RR^d),
\end{equation}
where $\mu\in\mathbb\RR^d$,
and where $\Sigma$ is a real $d\times d$  positive-definite matrix.
A \emph{Gaussian mixture} is a measure on $\RR^d$ that is a finite convex combination of Gaussian measures.
	
Mixtures of multivariate Gaussian distributions have several nice properties. 
In particular, in Titterington et al.\ \cite{TSM85}, 
it is shown that Gaussian kernel density estimators 
can approximate any continuous density given enough kernels (universal consistency). 
It is well known that Gaussian mixtures are  weak*-dense in
the space of all Borel probability measures on  $\RR^d$.
They also have numerous applications in statistics; 
for more on this, see Section~\ref{S:conclusion} below.
	
	
\subsection{A Cram\'er--Wold theorem for Gaussian mixtures}
	
In this section we consider the problem of testing for equality for two Gaussian mixtures
by looking at a finite number of projections.
The basic theorem underlying this approach has two ingredients. 
One is the abstract result Theorem~\ref{T:mixture}. 
The other is a characterization of Cram\'er--Wold systems for  Gaussian measures  in $\RR^d$
(and, more generally, for elliptical distributions) 
established in \cite[Theorems~1 and~2]{FMR23a}, which we now recall.
	
Let $S$ be a set of vectors in $\RR^d$. 
Then the corresponding set of lines $\{\langle x\rangle:x\in S\}$
is a Cram\'er--Wold system for the Gaussian measures in $\RR^d$
if and only if $S$ has the property that
the only real symmetric $d\times d$ matrix~$A$ satisfying $x^TAx=0$ for all $x\in S$
is the zero matrix.
A set $S$ with this  property is called a  \emph{symmetric-matrix uniqueness set} 
(or \emph{sm-uniqueness set} for short).
It was shown in \cite{FMR23a} that an sm-uniqueness set for $\RR^d$ spans $\RR^d$ 
and that it contains at least $(d^2+d)/2$ vectors.

We shall call $S$ a \emph{strong sm-uniqueness set} if every subset of $S$ containing $(d^2+d)/2$ vectors
is an sm-uniqueness set.  A method for generating examples of strong sm-uniqueness sets
will be described in Section~\ref{S:strongsmu} below.

We can now state our Cram\'er--Wold theorem for Gaussian mixtures.
	
\begin{theorem}\label{T:gaussmixture}
Let $P$ and $Q$ each be convex combinations of $m$ Gaussian measures on $\RR^d$. 
Let $S$ be a strong sm-uniqueness set for $\RR^d$ containing at least $(2m-1)(d^2+d-2)/2+1$ vectors.
If $P_{\langle x\rangle}=Q_{\langle x\rangle}$ for all $x\in S$, then $P=Q$.
\end{theorem}
	
\begin{proof}
We apply Theorem~\ref{T:mixture} with $\cP$ equal to the set of Gaussian measures on $\RR^d$
and $\cH:=\{\langle x\rangle: x\in S\}$. All we need to do is to check that $\cP$ and $\cH$ satisfy the
hypotheses \eqref{i:1} and \eqref{i:2} in Theorem~\ref{T:mixture}.
		
Concerning hypothesis~\eqref{i:1}, 
the projection of a multivariate Gaussian measure onto a line
is just a Gaussian measure on that line.
We have already seen in Theorem~\ref{T:normalindep}
that the set of all one-dimensional Gaussian measures is a
linearly independent family. So hypothesis~\eqref{i:1} holds.
		
As for hypothesis~\eqref{i:2}, we argue as follows.
If $S$ is partitioned into $2m-1$ sets, then one of them, $S_0$ say, 
must contain at least $(d^2+d)/2$ vectors
(otherwise $S$ would contain at most $(2m-1)(d^2+d-2)/2$ vectors, contrary to assumption).
As $S$ is a strong sm-uniqueness set, it follows that $S_0$ is an sm-uniqueness set.
In summary, if $S$ is partitioned into $2m-1$ sets, then at least one of them is an sm-uniqueness set.
In other words, if $\cH$ is partitioned into $2m-1$ sets,
then at least one of them is a Cram\'er--Wold system for the family of Gaussian measures. 
Thus hypothesis~\eqref{i:2} holds, and we are done.
\end{proof}
	
	
\subsection{A Cram\'er--Wold theorem for $t$-mixtures}
	
We now establish an analogue of Theorem~\ref{T:gaussmixture}
for mixtures of  multivariate $t$-dist\-ributions,
thereby allowing heavy-tailed distributions. 
A $t$-distribution on $\RR^d$ is a measure with density of the form
\[
f_{\nu,\mu,\Sigma}(x)
=c_{\nu,\mu,\Sigma}\Bigl(1+\frac{(x-\mu)^T\Sigma^{-1}(x-\mu)}{\nu}\Bigr)^{-(\nu+d)/2}\quad(x\in\RR^d),
\]
where $\nu$ is a positive integer, $\mu$ is a vector in $\RR^d$, 
and where $\Sigma$ is a positive-definite $d\times d$ matrix.
Once again, the constant $c_{\nu,\mu,\Sigma}$ is chosen 
to ensure that $\int_{\RR^d}f_{\nu,\mu,\Sigma}(x)\,dx=1$.
	
\begin{theorem}\label{T:tmixture}
Let $P,Q$ each be convex combinations of  $m$
multivariate  $t$-dist\-ributions on $\RR^d$. 
Let $S$ be a strong sm-uniqueness set for $\RR^d$ containing at least $(2m-1)(d^2+d-2)/2+1$ vectors.
If $P_{\langle x\rangle}=Q_{\langle x\rangle}$ for all $x\in S$, then $P=Q$.
\end{theorem}
	
\begin{proof}
This is virtually identical to the proof of Theorem~\ref{T:gaussmixture}.
Once again, we apply Theorem~\ref{T:mixture},
and we need to check that the hypotheses \eqref{i:1}
and \eqref{i:2} of that theorem hold.
		
Hypothesis~\eqref{i:1} holds because the one-dimensional
projection of a multivariate $t$-distribution is a univariate
$t$-distribution, and by Theorem~\ref{T:tlindep}
the family of all univariate $t$-distributions is a linearly independent set.
		
Hypothesis~\eqref{i:2} holds for the same reason that it did before.
Indeed, \cite[Theorem~1]{FMR23a} applies to all elliptical distributions, 
which includes $t$-dist\-ributions as well as Gaussian ones.
\end{proof}

	
\subsection{Strong sm-uniqueness sets}\label{S:strongsmu}
	
Theorems~\ref{T:gaussmixture} and \ref{T:tmixture} beg the question 
as to whether there exist strong sm-uniqueness sets of arbitrarily large cardinality.
The following result provides an affirmative answer,
and  suggests a realistic method for generating them.

\begin{theorem}\label{T:ssmu}
Let $d\ge2$, let $k\ge (d^2+d)/2$, and let 
\[
V:=\Bigl\{(v_1,\dots,v_k)\in (\RR^{d})^k: \{v_1,\dots,v_k\} \text{~is a strong sm-uniqueness set}\Bigr\}.
\]
Then $V$ is an open subset of $\RR^{dk}$, and $\RR^{dk}\setminus V$ has Lebesgue measure zero.
\end{theorem}

Thus, if $v_1,\dots,v_k$ are independent random vectors in $\RR^d$
with distributions given by densities on $\RR^d$, then, 
with probability one, the set $\{v_1,\dots,v_k\}$ 
is a strong sm-uniqueness set for $\RR^d$.
To test whether a specific family $\{v_1,\dots,v_k\}$ is a strong sm-uniqueness set,
one can use the following criterion for sm-uniqueness sets, 
which is also an ingredient in the proof of Theorem~\ref{T:ssmu}.

Let $d\ge2$ and set $D:=(d^2+d)/2$.
Given $x=(t_1,\dots,t_d)\in\RR^d$, 
let $\hat{x}$ be the upper triangular $d\times d$ matrix with entries 
$\hat{x}_{ij}:=t_it_j~(1\le i\le j\le d)$, 
but viewed as a column vector in $\RR^D$.

\begin{lemma}\label{L:smu}
A $D$-tuple $(x_1,\dots,x_D)$ of vectors in $\RR^d$ is an sm-uniqueness set 
if and only if $\hat{x}_1,\dots,\hat{x}_D$
are linearly independent vectors in $\RR^D$, 
in other words, if and only if the determinant of the $D\times D$ block matrix
$(\hat{x}_1|\hat{x}_2|\cdots|\hat{x}_D)$
does not vanish. 
\end{lemma}

\begin{proof}
This is \cite[Corollary~5]{FMR23a}.
\end{proof}

\begin{proof}[Proof of Theorem~\ref{T:ssmu}]
By Lemma~\ref{L:smu}, the set $\RR^{dk}\setminus V$ can be expressed as the union
\[
\bigcup_{1\le j_1<j_2<\dots <j_D\le k}\Bigl\{((v_1,\dots,v_k)\in(\RR^d)^k:\det\bigl(\hat{v}_{j_1}|\hat{v}_{j_2}|\cdots|\hat{v}_{j_D}\bigr)=0\Bigr\}.
\]
For each choice of $(j_1,j_2,\dots,j_D)$, the map 
$(v_1,\dots,v_k)\mapsto \det(\hat{v}_{j_1}|\hat{v}_{j_2}|\cdots|\hat{v}_{j_D})$ is a polynomial
in the entries of $(v_1,\dots,v_k)$ that is not identically zero, so its zero set is a closed subset 
of $\RR^{dk}$ of Lebesgue measure zero. As $\RR^{dk}\setminus V$ is a finite union of such sets,
it too is a  closed subset of $\RR^{dk}$ of Lebesgue measure zero.
\end{proof}


\subsection{Two examples}

In this subsection, we examine the
rationale for the number $(2m-1)(d^2+d-2)/2+1$ appearing in
Theorems~\ref{T:gaussmixture} and \ref{T:tmixture}.
In particular, we explain the quadratic growth in $d$ and the linear growth
in $m$.

The quadratic growth in $d$ is easily justified by the
following theorem,  essentially taken from \cite{FMR23a}.
We consider the special case $m=1$ 
(so that $P,Q$ are no longer mixtures). 
Thus $(2m-1)(d^2+d-2)/2+1=(d^2+d)/2$.

\begin{theorem}
Let $d\ge2$ and let $S$ be any set of vectors in $\RR^d$ containing strictly fewer than $(d^2+d)/2$ vectors. Then there exist Gaussian measures $P,Q$ on $\RR^d$ such that $P_{\langle x\rangle}=Q_{\langle x\rangle}$ for all $x\in S$, but $P\ne Q$.
\end{theorem}

\begin{proof}
By \cite[Proposition~2]{FMR23a}, the fact that $S$ has fewer
than $(d^2+d)/2$ elements implies that it is not an sm-uniqueness set. The existence of $P$ and $Q$ now follows from 
\cite[Theorem~2]{FMR23a}.
\end{proof}

This result explains not only the  presence of the factor
$(d^2+d)/2$ in Theorem~\ref{T:gaussmixture}, but also the need
for considering of the notion of sm-uniqueness set. 
We remark in passing that the quantity $(d^2+d)/2$ arises here because it is the number of degrees of freedom of a $d\times d$
covariance matrix. This is discussed in more detail in 
\cite[Section~2]{FMR23a}.

Now we turn to the dependence on $m$,
namely the number of Gaussian (or $t$-)
distributions making up the mixtures $P$ and $Q$.
In our first example, we take $d=2$ and $m=2$.
Though apparently modest, we believe that this example 
is instructive.

\begin{example}\label{Ex:m=2}
Let $S$ be the subset of $\RR^2$ defined by
\[
S:=\Bigl\{
{1\choose0},{0\choose1},{1\choose1},{1\choose-1}\Bigr\}.
\]
Then $S$ is a strong sm-uniqueness set,
and there exist measures $P,Q$ on $\RR^2$, 
each a convex combination of two Gaussian measures, 
such that $P_{\langle x\rangle}=Q_{\langle x\rangle}$
for all $x\in S$, but $P\ne Q$.
\end{example}

\begin{proof}
First of all, we remark that, by \cite[Corollary~3]{FMR23a},
a set of three vectors in $\RR^2$ is an sm-uniqueness set
provided that no two vectors in the set are multiples of one another.
It follows that $S$ is a strong sm-uniqueness set.

Now set 
\[
\Sigma_1:=
\begin{pmatrix}
2 &0\\
0 &1
\end{pmatrix},
\quad
\Sigma_2:=
\begin{pmatrix}
1 &1\\
1 &2
\end{pmatrix},
\quad
\Sigma_1':=
\begin{pmatrix}
2 &1\\
1 &1
\end{pmatrix},
\quad
\Sigma_2':=
\begin{pmatrix}
1 &0\\
0 &2
\end{pmatrix}.
\]
Let $P:=(P_1+P_2)/2$, where $P_j$ is Gaussian
with mean $0$ and covariance $\Sigma_j$,
and let $Q:=(Q_1+Q_2)/2$, where $Q_j$ is Gaussian
with mean $0$ and covariance $\Sigma_j'$.
Clearly $P\ne Q$.

Elementary calculations show that, if 
$x={1\choose0}$ or ${0\choose1}$, then
\[
x^T\Sigma_1 x=x^T\Sigma_1'x
\quad\text{and}\quad
x^T\Sigma_2 x=x^T\Sigma_2'x,
\]
while if $x={1\choose1}$ or ${1\choose-1}$, then
\[
x^T\Sigma_1 x=x^T\Sigma_2'x
\quad\text{and}\quad
x^T\Sigma_2 x=x^T\Sigma_1'x.
\]
Note that, in this last display, $\Sigma_1'$ and $\Sigma_2'$ have changed places. However, since $P$ and $Q$ are averages of the $P_j$
and $Q_j$ respectively, we do have
$P_{\langle x\rangle}=Q_{\langle x\rangle}$ for all $x\in S$.
\end{proof}

If $d=2$ and $m=2$, then $(2m-1)(d^2+d-2)/2+1=7$.
Example~\ref{Ex:m=2} shows that, in Theorem~\ref{T:gaussmixture},
this number cannot be replaced by $4$. However, calculations
similar to those carried out in Example~\ref{Ex:m=2} indicate that $5$ lines do suffice. Thus, in this case at least, the 
bound $(2m-1)(d^2+d-2)/2+1$  is not sharp.

Despite the lack of sharpness, the linear growth in $m$
is correct. To see this, we consider a second example 
where,
instead of varying the covariance matrices, we vary the means.
This example is inspired by one for discrete measures presented in \cite{He56}.

\begin{example}\label{Ex:m>1}
Let $m\ge2$ and define a subset $S$ of $\RR^2$ by
\[
S:=\Bigl\{\begin{pmatrix}
\cos(j\pi/m)\\
\sin(j\pi/m)
\end{pmatrix}: 
j=1,\dots,m
\Bigr\}.
\]
Then $S$ is a strong sm-uniqueness set,
and there exist measures $P,Q$ on $\RR^2$, each a convex combination
of $m$ Gaussian measures, such that $P_{\langle x\rangle}=Q_{\langle x\rangle}$ for all $x\in S$, but $P\ne Q$.
\end{example}

\begin{proof}
As in the previous example,
a set of three vectors in $\RR^2$ is an sm-uniqueness set
provided that no two vectors in the set are multiples of one another.
It follows that $S$ is a strong sm-uniqueness set.

For $j=1,\dots,m$, set 
\[
\mu_j:=
\begin{pmatrix}
\cos((2j-1/2)\pi/m)\\
\sin((2j-1/2)\pi/m)
\end{pmatrix}
\quad\text{and}\quad
\mu_j':=
\begin{pmatrix}
\cos((2j+1/2)\pi/m)\\
\sin((2j+1/2)\pi/m)
\end{pmatrix}.
\]
Thus the sequence $\mu_1,\mu_1',\mu_2,\mu_2',\dots,\mu_m,\mu_m'$ traces out, in order, the vertices of a  regular $(2m)$-gon
centred at the origin. The lines defined by the set $S$ are precisely
the bisectors of this $(2m)$-gon passing through the midpoints of the edges.
For $j=1,\dots,m$, let $P_j,Q_j$ be  Gaussian measures with means at $\mu_j,\mu_j'$ respectively, all of them having covariance matrix equal to the identity. Finally, let
$P=(1/m)\sum_{j=1}^mP_j$ and $Q:=(1/m)\sum_{j=1}^mQ_j$.
Then it is geometrically clear that $P_{\langle x\rangle}=Q_{\langle x\rangle}$ for all $x\in S$, but evidently $P\ne Q$.
\end{proof}

This example works because of the symmetry
of the set-up. Though such symmetry is rather exceptional, there is
nothing in Theorem~\ref{T:gaussmixture} to exclude this
possibility, and so it needs to be taken into account
when deciding the number of lines to be used. 
Example~\ref{Ex:m>1} shows that (when $d=2$), 
this is number has to be 
at least $m+1$. This explains the linear growth in $m$
in the bound $(2m-1)(d^2+d-2)/2+1$.

	
\section{Concluding remarks}\label{S:conclusion}

Gaussian-mixture models have been shown to be very effective in modeling different real data, see for instance Titterington et al.\ \cite{TSM85} for a deep study of their properties. They provide flexible and general models, and relevant applications can be found in the literature in different fields like density estimation, machine learning and clustering,  among others. The estimation of these models is quite involved, in particular for high-dimensional data, 
typically using Markov-chain Monte-Carlo methods in a Bayesian framework.
The Cram\'er--Wold device has a role to play in this circle of ideas, in particular through
Theorems~\ref{T:gaussmixture} and \ref{T:tmixture}. 
This is explored in detail in the companion paper \cite{FMR25}.	

\bibliographystyle{amsplain}
\bibliography{CWmixtures}

\end{document}